\documentclass[a4paper,11pt,oneside]{article} 
\usepackage{amsfonts}
\usepackage{amssymb}
\usepackage{indentfirst}
\usepackage{caption}
\usepackage{mathrsfs}
\usepackage{amsfonts,amsmath,latexsym,bm}
\usepackage{amsthm}
\usepackage{float}
\usepackage{epsfig}
\usepackage{graphicx}
\newcommand{\tabcaption}{\def\@captype{table}\caption}
\newcommand{\norm}[1]{\left\lVert#1\right\rVert}
\newcommand{\vertiii}[1]{|#1|}
\newcommand{\nvertiii}[1]{{\left\vert\kern-0.25ex\left\vert\kern-0.25ex\left\vert #1 \right\vert
                \kern-0.25ex\right\vert\kern-0.25ex\right\vert}}
\newcommand{\bmnabla}{\bm\nabla}
\newcommand{\bint}[2]{\langle#1\rangle_{#2}}
\newcommand{\lam}{\lambda}
\newcommand{\mysum}{\sum\limits}
\newtheorem{lem}{Lemma}[section]
\newtheorem{thm}{Theorem}[section]

\newtheorem{rem}{Remark}[section]
\newtheorem{assumption}{Assumption}

\newtheorem{algorithm}{Algorithm}

\numberwithin{equation}{section}
\begin{document}
\title
{
    \Large\bf A two-level algorithm for  the weak Galerkin discretization of diffusion problems
    \thanks
    {
        This work was supported by National Natural Science Foundation of China (11171239),  Major Research 
	Plan of  National Natural Science Foundation of China (91430105) and Open Fund of  Key Laboratory of Mountain 
	Hazards and Earth Surface Processes, CAS. 
    }
}

\author 
{
    Binjie Li\thanks{Email: libinjiefem@yahoo.com}, \quad
    Xiaoping Xie \thanks{Corresponding author. Email: xpxie@scu.edu.cn}\\
    {School of Mathematics, Sichuan University, Chengdu 610064, China}
}

\date{}
\maketitle
\begin{abstract}
    This paper analyzes a two-level  algorithm for the weak Galerkin (WG) finite element methods 
    based on local Raviart-Thomas (RT) and Brezzi-Douglas-Marini (BDM) mixed elements for two- and three-dimensional 
    diffusion problems with Dirichlet condition.  We first show the condition numbers of the stiffness matrices 
    arising from the WG methods are of $O(h^{-2})$. We use an extended version of the Xu-Zikatanov (XZ) identity 
    to derive the convergence of the   algorithm  without any regularity assumption. Finally we provide some numerical results.
\end{abstract}
\renewcommand{\theequation}{\thesection.\arabic{equation}}

\section{Introduction}
Let $\Omega\subset R^d~(d=2,3)$ be a bounded polyhedral domain. Consider the following diffusion problem:
\begin{equation}\label{eq_problem}
    \left\{
        \begin{array}{rcll}
            -\text{div}({\bm a}\bmnabla u) &= & f & \text{in $\Omega$},\\
            u &=& 0 & \text{on $\partial\Omega$},
        \end{array}
    \right.
\end{equation}
where ${\bm a}\in [L^{\infty}(\Omega)]^{d\times d}$ is a given  symmetric positive-definite permeability
tensor, $f\in L^2(\Omega)$.\par

The weak Galerkin(WG) finite element method was first introduced and analyzed by Wang and Ye 
\cite{WangYe2013} for general second order elliptic problems and later developed by their research group in 
\cite{Mu-Wang-Ye1, Mu-Wang-Ye2,Mu-Wang-Wang-Ye, Mu-Wang-Ye-Zhao, Mu-Wang-Wei-Ye,WGSTOKES, WGBIHARMONIC}. 
It is designed by using a weakly defined gradient operator over 
functions with discontinuity. The method, based on local Raviart-Thomas (RT) elements \cite{RT} or 
Brezzi-Douglas-Marini (BDM) elements \cite{BrezziDouglasMarini1985}, allows the use of totally 
discontinuous piecewise polynomials in the finite element procedure, as is common in discontinuous 
Galerkin methods \cite{Arnold-Brezzi-Cockburn-Marini} and hybridized discontinuous Galerkin methods 
\cite{ Cockburn-GOPALAKRISHNAN-LAZAROV}.
As shown  in \cite{WangYe2013, 
    Mu-Wang-Ye1, Mu-Wang-Ye2,Mu-Wang-Wang-Ye,Mu-Wang-Wei-Ye}, the WG method also enjoys an easy-to-implement 
formulation that  inherits the physical property of mass conservation locally on each element.  
We note that when $\bm a$ in (\ref{eq_problem}) is a piecewise-constant matrix, the WG method, 
by introducing the discrete weak gradient as an independent variable, is equivalent to some  
hybridized version of the corresponding mixed RT or BDM method \cite{ArnoldBrezzi1985, 
    BrezziDouglasMarini1985} (cf. Remark \ref{remark-1}).

As one knows, multigrid methods are among the most efficient methods for solving linear algebraic 
systems arising from the discretization of partial differential equations. By now, the research of the
multigrid methods for second order elliptic problems has reached a mature stage in some sense
(see \cite{Ban1.,Ban2.,Br.,Bpj3.,Bpj5,Bpj1,Bp1.,B2,H1.,Wu-Chen-Xie-Xu,Xu1992,Xu1996, 
    Xu1997, Xu-Z2002,XU-CHEN} and the references therein).  Especially, Xu,  Chen, and Nochetto \cite{XU-CHEN} presented
an overview of the multigrid methods   in an elegant fashion. For the model problem (\ref{eq_problem}), Brenner \cite{Brenner.S1992} 
developed an optimal order multigrid method for the lowest-order Raviart-Thomas mixed triangular 
finite element. The algorithm and the convergence analysis are based on the equivalence between 
Raviart-Thomas mixed methods and certain nonconforming methods. 
In \cite{GOPALAKRISHNAN2009}
Gopalakrishnan and Tan analyzed the convergence of a variable V-cycle multigrid algorithm for the 
hybridized mixed method for Poisson problems. Following the same idea, Cockburn et al.  \cite{Cockburn.B;2013} 
analyzed the convergence of a  non-nested multigrid V-cycle  algorithm, with a single smoothing step 
per level, for one type of HDG method. One may refer to \cite{mg_p1, Brenner.S1992, Brenner1999, Brenner2004,DOBREV;2006,
    Duan;2007,GOPALAKRISHNAN;2003, KRAUS;2008, KRAUS;2008b} for multigrid algorithms for nonconforming 
and DG methods.

This paper is to analyze a two-level algorithm for the WG methods.  We show 
the condition numbers of the WG  systems  are of $O(h^{-2})$.  We follow the basic ideas of 
\cite{Xu-Z2002, XU-CHEN,Chen_XZ}  to  establish  an extended version of the Xu-Zikatanov (XZ) identity \cite{Xu-Z2002}, 
and then   derive the convergence of the algorithm without any regularity assumption.

The rest of this paper is organized as follows. Section \ref{sec_def_WG}  introduces the WG methods. 
Section \ref{sec_estimate}  analyzes the conditioning  of the WG systems. Section \ref{sec_multigrid} describes the two-level algorithm, and analyzes  its convergence. Section \ref{sec_numerical} provides some numerical 
experiments to verify  our theoretical results.\par  
\section{Weak Galerkin finite element method}\label{sec_def_WG}
\subsection{Preliminaries and Notations}\label{notations}
Throughout this paper, we shall use the standard definitions of Sobolev spaces and their 
norms(\cite{ADAMS}), namely, for an arbitrary open set, $D$, of $\mathbb R^d$ and any 
nonnegative integer $s$, 
\begin{displaymath}
    \begin{array}{rcl}
        H^s(D) &:=& \{v\in L^2(D):\partial^{\alpha}v\in L^2(D),\forall|\alpha|\leqslant s\},\\
        \norm{v}_{s,D} &:=& \left(\mysum_{0\leq j\leq s}|v|_{j,D}^2\right)^{\frac{1}{2}}, 
        \quad 
        |v|_{j,D}:=\left(\mysum_{|\alpha|= j}\int_{D}|\partial^{\alpha}v|^2\right)^{\frac{1}{2}}.
    \end{array}
\end{displaymath}
We use $(\cdot, \cdot)_D$ and $\bint{\cdot,\cdot}{\partial D}$ to denote the standard $L^2$ 
inner products on $L^2(D)$ and $L^2(\partial D)$, respectively, and   use $\norm{\cdot}_D$ 
and $\norm{\cdot}_{\partial D}$ to denote the norms induced by $(\cdot,\cdot)_D$ and 
$\bint{\cdot,\cdot}{\partial D}$, respectively. In particular, $\norm{\cdot}$ abbreviates 
$\norm{\cdot}_\Omega$.
\par
Let $\mathcal T_h$ be a regular triangulation of $\Omega$. For any $T\in\mathcal T_h$, we 
denote by $h_T$ the diameter of $T$ and set $h:=\max\limits_{T\in\mathcal T_h}h_T$. We denote by
$\mathcal F_h$    the set of all faces of $\mathcal T_h$. 
\par
We introduce some mesh-dependent inner products and mesh-dependent norms as follows. 
We define $\bint{\cdot,\cdot}{h}:L^2( \mathcal F_h )\times L^2(\mathcal F_h )\to\mathbb R$ by
\begin{equation}\label{def_inner_products_Mh}
    \bint{\lam,\mu}{h}:= \sum_{T\in\mathcal T_h}h_T\int_{\partial T}\lam\mu,
    ~\forall\lam,\mu\in L^2(\mathcal F_h ),
\end{equation}
and $(\cdot,\cdot)_h:[L^2(\Omega)\times L^2(\mathcal F_h )]\times[L^2(\Omega)\times L^2(\mathcal F_h )]
\to\mathbb R$   by
\begin{equation}\label{def_inner_product_Wh}
    ((u,\lam),(v,\mu))_h:=(u,v)_{\Omega} + \bint{\lam,\mu}{h},~\forall (u,\lam),(v,\mu)\in 
    L^2(\Omega)\times L^2(\mathcal F_h ).
\end{equation}
With a little abuse of notations, we use $\norm{\cdot}_h$ to denote the norms induced by the 
inner products $\bint{\cdot,\cdot}{h}$ and $(\cdot,\cdot)_h$, i.e.,
\begin{eqnarray}
    \norm{\mu}_h &:=& \bint{\mu,\mu}{h}^{\frac{1}{2}},
    ~~~~~~~~~~~\forall\mu\in L^2(\mathcal F_h ),\label{def_norm_Mh}\\
    \norm{(v,\mu)}_h &:=& ((v,\mu),(v,\mu))_h^{\frac{1}{2}}=\left(\norm{v}^2+\norm{\mu}_h^2\right)^{\frac{1}{2}},
    ~\forall (v,\mu)\in L^2(\Omega)\times L^2(\mathcal F_h ).\label{def_norm_Wh}
\end{eqnarray}

We also need the following  elementwise norm and seminorms:  for any $\mu \in L^2(\mathcal F_h )$, 
\begin{displaymath}
    \norm{\mu}_{h,\partial T} := h_T^{\frac{1}{2}}\norm{\mu}_{\partial T},
\end{displaymath}
\begin{displaymath}
    \begin{array}{rcl}
        |\mu|^2_{h,\partial T} := h_T^{-1}\norm{\mu-m_T(\mu)}^2_{\partial T}\quad \text{with}\quad
        m_T(\mu) := \frac{1}{d+1}\sum\limits_{F\in\mathcal F_T}\frac{1}{|F|}\int_F\mu,\\
    \end{array}
\end{displaymath}
and
\begin{equation}\label{semi-norm}
    |\mu|_h: = (\sum_{T\in\mathcal T_h}|{\mu}|_{h, \partial T}^2)^{\frac{1}{2}}, 
\end{equation}
where $\mathcal F_T$ denotes the set of all faces of $T$, and $|F|$ denotes the (d-1)-dimensional 
Hausdorff measure of $F$. 

\par
Throughout this paper, $x\lesssim y~( x \gtrsim y)$ means $x\leqslant Cy~(x\geq Cy)$, 
where $C$ denotes a positive constant that is independent of the mesh size $h$. The notation 
$x \sim y$ abbreviates $x \lesssim y\lesssim x$.
\subsection{Weak Galerkin  formulations}
We first introduce two spaces:
\begin{displaymath}
    \begin{array}{rcl}
        V_h &:=& \{v_h\in L^2(\Omega):v_h|_T\in V(T),~\forall T\in\mathcal T_h\},\\
        M_h^0 &:=& \{\mu_h\in L^2(\mathcal F_h ):\mu_h|_F\in M(F),~\forall F\in\mathcal F_h,
        \mu_h|_{\partial\Omega}=0\},
    \end{array}
\end{displaymath}
where $V(T)$ and $M(F)$ denote  two local finite dimensional spaces.  

For $T\in\mathcal T_h$, let $\bm W(T)$ be a local finite dimensional
subspace of $[L^2(T)]^d$.
Then, following \cite{WangYe2013}, we introduce the discrete weak gradient $\bmnabla_w:L^2(T)\times L^2(\partial T)\to\bm W(T)$ defined by
\begin{equation}\label{grad-v-mu}
    \bmnabla_w(v,\mu) = \bmnabla_w^iv + \bmnabla_w^b\mu,~\forall (v,\mu)\in L^2(T)\times L^2(\partial T),
\end{equation}
where 
  $\bmnabla_w^iv, \bmnabla_w^b\mu\in \bm W(T)$  satisfy,   for any $\bm q\in\bm W(T)$, 
\begin{equation}\label{grad-v}
    (\bmnabla_w^iv,\bm q)_T = -(v,div~\bm q)_T,
\end{equation}
\begin{equation}\label{grad-mu}
    (\bmnabla_w^b\mu,\bm q)_T = \bint{\mu,\bm q\cdot\bm n}{\partial T}.
\end{equation}
\par
The WG method for problem \eqref{eq_problem} reads as follows(\cite{WangYe2013}): 
Seek $(u_h, \lambda_h) \in V_h\times M_h^0$ such that
\begin{equation}\label{discretization1}
    a_h((u_h,\lambda_h),(v_h,\mu_h))_{\Omega} = (f, v_h)_{\Omega},~\forall (v_h,\mu_h) \in V_h\times M_h^0,
\end{equation}
where
\begin{displaymath}
    a_h((u_h,\lambda_h),(v_h,\mu_h)): = (\bm a\bmnabla_w(u_h,\lambda_h),\bmnabla_w(v_h,\mu_h))_{\Omega}.
\end{displaymath}
\par
For any set $D$,  we denote by $P_j(D)$ the set of polynomials of degree $\leq j$ on $D$.   
This paper considers two type of WG methods \cite{WangYe2013} which are based on local RT 
and BDM mixed elements, respectively:
\begin{description}
    \item[Type 1. ] $V(T) = P_k(T)$, $M(F)= P_k(F)$, $\bm W(T)=[P_k(T)]^d+P_k(T)\bm x$.
    \item[Type 2. ] $V(T) = P_{k-1}(T)$, $M(F)= P_k(F)$, $\bm W(T) = [P_k(T)]^d$ ($k\geq1$).
\end{description}
\begin{rem}\label{remark-1}
    When $\bm a$ is a piecewise constant matrix, we can show that the two type of WG methods are 
    equivalent to the hybridized version of the corresponding mixed RT and BDM method 
    (\cite{ArnoldBrezzi1985,BrezziDouglasMarini1985}) respectively. 
    In fact,  by introducing the vector
    $\bm p_h :=\bm a \bmnabla_w(u_h,\lambda_h)$ and the space $ \bm W_h:=\{\bm q_h\in [L^2(\Omega)]^d:\bm q_h|_T\in\bm W(T)\},$ 
    it's straightforward that the WG scheme \eqref{discretization1} is equivalent to the following problem: 
    Seek $(\bm p_h,u_h,\lambda_h)\in\bm W_h\times V_h\times M_h^0$, such that
    \begin{displaymath}
        \begin{array}{rcll}
            (\bm a^{-1}\bm p_h,\bm q_h)_\Omega + \sum\limits_{T\in\mathcal T_h}(u_h,\text{div} \bm q_h)_T 
            - \sum\limits_{T\in\mathcal T_h}\bint{\lam_h,\bm q_h\cdot\bm n}{\partial T} &=& 0,\\ 
            - \sum\limits_{T\in\mathcal T_h}(v_h,\text{div} \bm p_h)_T &=& (f, v_h)_\Omega, \\
            \sum\limits_{T\in\mathcal T_h}\bint{\bm p_h \cdot \bm n, \mu_h}{\partial T}&=& 0
        \end{array}
    \end{displaymath}
    hold for all $(\bm q_h,v_h,\mu_h)\in \bm W_h\times V_h\times M_h^0$.
    This scheme is no other than the hybridized version of the RT  mixed element method (cf. (1.18)
    in \cite{ArnoldBrezzi1985}) or the  BDM mixed method (cf. (1.13) in \cite{BrezziDouglasMarini1985}). 
\end{rem}

In the following we give an operator form and a matrix form of the WG discretization \eqref{discretization1}. 
Let  $\{\phi_i:i=1,2,\ldots, M\}\subset V_h$ and $\{\eta_i:i=1,2,\ldots, N\} \subset M_h^0$ be nodal bases for   
$V_h$ and $M_h^0$, respectively. Denote  by $\tilde u_h, \tilde v_h\in \mathbb R^M$ the vectors of coefficients 
of $u_h$, $v_h$ in the $\{\phi_i\}$-basis, and by $\tilde\lambda_h, \tilde\mu_h\in \mathbb R^N$ the vectors of
coefficients of $\lambda_h,\mu_h$ in the $\{\eta_i\}$-basis, respectively. 

Define the operators $\mathcal C_h:V_h\to V_h$, $\mathcal B_h:V_h\to M_h^0$, $\mathcal B_h^t: M_h^0\to V_h$, 
$\mathcal D_h:M_h^0\to M_h^0$, and the matrices $B_h\in\mathbb R^{N\times M},C_h\in\mathbb R^{M\times M}$ ,
$D_h\in\mathbb R^{N \times N}$ respectively by
\begin{displaymath}
    \begin{array}{rcll}
        (\mathcal C_hu_h, v_h)_{\Omega}
        &:=&(\bm a\bmnabla^i_w u_h,\bmnabla^i_w v_h)_{\Omega}=:\tilde u_h^T C_h \tilde v_h,&\forall u_h, v_h \in V_h,\\
        \langle\mathcal B_h u_h, \lam_h\rangle_h &:=&(\bm a\bmnabla^i_w u_h, \bmnabla^b_w\lam_h)_{\Omega} 
        =:(u_h,\mathcal B_h^t \lam_h)_\Omega=:\tilde u_h^T B_h^T \tilde \lam_h,
        &\forall u_h\in V_h,\lambda_h \in M_h^0,\\
        \langle\mathcal D_h\lam_h,\mu_h\rangle_h &:=&(\bm a\bmnabla^b_w\lam_h,\bmnabla^b_w\mu_h)_{\Omega}
        =:\tilde \lam_h^T D_h^T \tilde \mu_h,&\forall\lam_h,\mu_h\in M_h^0.\\
    \end{array}
\end{displaymath}
Let $\mathcal A_h:V_h\times M_h^0\to V_h\times M_h^0$ and $A_h\in\mathbb R^{(M+N)\times(M+N)}$ be defined by
\begin{equation}\label{def_A}
    \begin{array}{l}
        (\mathcal A_h(u_h,\lam_h),(v_h,\mu_h))_h
        :=a_h((u_h,\lam_h),(v_h,\mu_h))=:
        (\tilde u_h^t~\tilde\lam_h^t)A_h\left(\begin{array}{c}\tilde v_h\\\tilde\mu_h\end{array}\right)
    \end{array}
\end{equation}
for any $ (u_h,\lam_h),(v_h,\mu_h)\in V_h\times M_h^0$. Then we have
\begin{equation}
    \mathcal A_h = \left(
        \begin{array}{cc}
            \mathcal C_h & \mathcal B_h^t \\
            \mathcal B_h & \mathcal D_h \\
        \end{array}
    \right),\quad 
    A_h = \left(
        \begin{array}{cc}
            C_h & B_h^T \\
            B_h & D_h \\
        \end{array}
    \right),
\end{equation}
and the WG discretization \eqref{discretization1} is equivalent to the following system:
Seek $(u_h,\lam_h)\in V_h\times M_h^0$ such that
\begin{equation}\label{operator system}
    \mathcal A_h(u_h,\lam_h) = b_h
\end{equation}
with $b_h:=(f_h,0)$ and    $f_h\in V_h$  denoting the standard $L^2-$orthogonal projection of $f$ onto $V_h$.

\section{Conditioning of   WG methods}\label{sec_estimate}
In what follows we assume $\mathcal T_h$ to be a quasi-uniform triangulation. We recall that   
$\norm{\cdot}_h$, $|{\cdot}|_h$, $\norm{\cdot}_{ T}$, $\norm{\cdot}_{h,\partial T}$, and 
$|{\cdot}|_{h,\partial T}$ are defined in Subsection \ref{notations}.

We first present a basic estimate as follows.
\begin{lem}\label{ref_lem_basic_inequality}
    For any $\mu_h\in M^0_h$, it holds
    \begin{equation}\label{basic_a}
        \norm{\mu_h}_h \lesssim |{\mu_h}|_h.
    \end{equation}
\end{lem}
\begin{proof}	
    See Appendix \ref{append_2}.
\end{proof}

 For any simplex $T$, define 
 $$M(\partial T):=\{\mu\in L^2(\partial T):\mu|_F\in M(F),
    \text{ for each face $F$ of $T$}\}.$$
The following lemma  gives some basic estimates of weak gradients. 

\begin{lem}\label{lem_basic_inequalities}
    For any any $T\in\mathcal T_h$ and $(v,\mu)\in V(T)\times M(\partial T)$, it holds
    \begin{subequations}
        \begin{eqnarray}
            \norm{\bmnabla^i_w v}_T &\sim& h_T^{-1}\norm{v}_T,\label{basic_b}\\
            \norm{\bmnabla^b_w \mu}_T &\sim& h_T^{-1}\norm{\mu}_{h, \partial T},\label{basic_c}\\
            \norm{\bmnabla_w(v,\mu)}_T &\sim& h_T^{-1}\norm{v-m_T(\mu)}_T + \vertiii{\mu}_{h,\partial T}.\label{basic_d}
        \end{eqnarray}
    \end{subequations}
\end{lem}
\begin{proof}
    See Appendix \ref{append_1}.
\end{proof}
In view of Lemmas \ref{ref_lem_basic_inequality}-\ref{lem_basic_inequalities}, we have the following conclusion. 
\begin{thm}\label{bound1}
    For any $(v_h,\mu_h)\in V_h\times M_h^0$, it holds
    \begin{equation}\label{left-right}
        \|(v_h,\mu_h)\|_h^2 \lesssim a_h((v_h,\mu_h),(v_h,\mu_h))\lesssim h^{-2}\|(v_h,\mu_h)\|_h^2.
    \end{equation}
\end{thm}
\begin{proof}
    From \eqref{basic_d} it follows
    \begin{equation}\label{lxy}
        \norm{\bmnabla_w(v_h,\mu_h)}^2 \sim \sum_{T\in\mathcal T_h}h_T^{-2}\norm{v_h-m_T(\mu_h)}_T^2 + \vertiii{\mu_h}_h^2.
    \end{equation}
    Since
    \begin{displaymath}
        |m_T(\mu_h)| 
        \leqslant \frac{1}{d+1}\sum_{F\in\mathcal F_T}|\frac{1}{|F|}\int_F\mu_h| 
        \lesssim h_T^{-\frac{d-1}{2}}\norm{\mu_h}_{\partial T},
    \end{displaymath}
    we have
    \begin{equation}\label{estimate_m_T}
        \norm{m_T(\mu_h)}_T 
        \lesssim h_T^{\frac{1}{2}}\norm{\mu_h}_{\partial T}
        \lesssim \norm{\mu_h}_{h, \partial T},
    \end{equation}
    which, together with Lemmas \ref{ref_lem_basic_inequality}-\ref{lem_basic_inequalities}, implies
    \begin{eqnarray}\label{thm_conditon_1}
        \norm{v_h}^2 
        &\lesssim& \sum_{T\in\mathcal T_h}\left\{\norm{v_h-m_T(\mu_h)}_T^2+\norm{m_T(\mu_h)}_T^2\right\}\nonumber\\
        &\lesssim& \sum_{T\in\mathcal T_h}\norm{v_h-m_T(\mu_h)}_T^2+\norm{\mu_h}_h^2\nonumber\\
        &\lesssim& \sum_{T\in\mathcal T_h}\norm{v_h-m_T(\mu_h)}_T^2+|{\mu_h}|_h^2.
    \end{eqnarray}
    A combination of  \eqref{basic_a}, \eqref{lxy} and \eqref{thm_conditon_1} yields
    \begin{equation}\label{left}
        \|(v_h,\mu_h)\|_h^2=\|v_h\|^2+\|\mu_h\|_h^2\lesssim a_h((v_h, \mu_h), (v_h, \mu_h)). 
    \end{equation}
    On the other hand,  it holds
    \begin{eqnarray}\label{right}
        a_h((v_h,\mu_h),(v_h,\mu_h)) 
        &\lesssim& \norm{\bmnabla_w v_h}^2 + \norm{\bmnabla_w\mu_h}^2\nonumber\\
        &\lesssim& h^{-2}\norm{v_h}^2+h^{-2}\norm{\mu_h}_h^2
        ~~~\text{by \eqref{basic_b} and \eqref{basic_c}}\nonumber\\
        &\lesssim& h^{-2}\norm{(v_h,\mu_h)}_h^2.
    \end{eqnarray}
    The estimates (\ref{left})-(\ref{right}) lead to the desired result (\ref{left-right}).
\end{proof}

\begin{thm}\label{bound2}
    It holds
    \begin{equation}\label{eq_low}
        \sup_{(v_h, \mu_h)\in V_h\times M_h^0}\frac{a_h((v_h,\mu_h),(v_h,\mu_h))}{\|(v_h,\mu_h)\|_h^2}\gtrsim h^{-2}.
    \end{equation}
   In addition,   
    \begin{equation}\label{eq_up}
        \inf_{(v_h,\mu_h)\in V_h\times M_h^0}\frac{a_h(v_h,\mu_h),(v_h,\mu_h))}{\|(v_h,\mu_h)\|_h^2}\lesssim 1
    \end{equation}
    holds if $h$ is sufficiently small.
\end{thm}
\begin{proof}
    Given $v_h\in V_h$, from Lemma \ref{lem_basic_inequalities} it follows
    \begin{equation}
        a_h((v_h, 0), (v_h, 0)) \sim h^{-2}\norm{v_h}^2,
    \end{equation}
    which implies \eqref{eq_low}.\par
    Let $s$ be the smallest eigenvalue of problem \eqref{eq_problem} with $f=su$ and let 
    $u_0\in H^1_0(\Omega)$ be the corresponding eigenvector function. Then it holds 
    \begin{equation}\label{eigenvec}
    \norm{\bmnabla u_0}^2\sim s\norm{u_0}^2.
    \end{equation}
    
    In the analysis below, we shall denote by $C$
      a positive constant that is independent of the mesh size $h$ and  may take
    a different value at its each occurrence.\par 
    
    We define $(v_h,\mu_h)\in V_h\times M_h^0$ by
    \begin{displaymath}
        \begin{array}{rcll}
            v_h|_T &=& m_T(u_0), &\forall T\in\mathcal T_h,\\
            \mu_h|_F &=&\frac{1}{|F|}\int_Fu_0,&\forall F\in\mathcal F_h.
        \end{array}
    \end{displaymath}
   By the definition of $m_T(\cdot)$ it is easy to see 
   \begin{equation}\label{331}
   m_T(\mu_h) = \frac{1}{d+1}\sum\limits_{F\in\mathcal F_T}\frac{1}{|F|}\int_F\mu_h=m_T(u_0).
   \end{equation}
    Standard scaling arguments yield
    \begin{eqnarray}
        \norm{u_0-m_T(u_0)}_T &\lesssim& h_T|u_0|_{1,T},\label{11}\\
        \norm{\mu_h-m_T(\mu_h)}_{\partial T}&\lesssim& h_T^{\frac{1}{2}}|u_0|_{1,T}.\label{22}
    \end{eqnarray}
    Thus, in view of \eqref{11} and \eqref{eigenvec} we have 
    \begin{equation}\label{33}
        \begin{split}
            \norm{v_h}^2
            &=\sum_{T\in\mathcal T_h}\norm{m_T(u_0)}^2_T\geqslant\sum_{T\in\mathcal T_h}\left\{\frac{1}{2}\norm{u_0}^2_T - 
                \norm{u_0-m_T(u_0)}^2_T\right\}\\
            &\gtrsim\sum_{T\in\mathcal T_h}\left\{\norm{u_0}^2_T-Ch_T^2|u_0|^2_{1,T}\right\}\\
            &\gtrsim(1-sCh^2)\norm{u_0}^2,
        \end{split}
    \end{equation}
   which, together with \eqref{22} and \eqref{331},  further implies
    \begin{equation}\label{44}
        \begin{split}
            \norm{\mu_h}^2_h
            &\geqslant\sum_{T\in\mathcal T_h}h_T\left(\frac{1}{2}\norm{m_T(\mu_h)}^2_{\partial T} -
                \norm{\mu_h-m_T(\mu_h)}^2_{\partial T}\right)\\
            &\gtrsim\sum_{T\in\mathcal T_h}h_T\norm{m_T(\mu_h)}^2_{\partial T}-Ch^2|u_0|^2_{1,\Omega}\\
            &\gtrsim\sum_{T\in\mathcal T_h}\norm{m_T(u_0)}^2_T - Ch^2|u_0|^2_{1,\Omega}
         \\
            &\gtrsim (1-sCh^2)\norm{u_0}^2.
        \end{split}
    \end{equation}
   On the other hand,  from the definition \eqref{semi-norm} and the estimate \eqref{22} it follows
    \begin{equation}\label{55}
        |\mu_h|_h \lesssim |u_0|_{1,\Omega}.
    \end{equation}
    Therefore,	it holds
    \begin{displaymath}
        \begin{split}
            \frac{a_h((v_h,\mu_h),(v_h,\mu_h))}{\norm{(v_h,\mu_h)}^2_h}
            &\sim\frac{\norm{\bmnabla_w(v_h,\mu_h)}^2}{\norm{(v_h,\mu_h)}_h^2}\\
            &\sim\frac{|\mu_h|^2_h}{\norm{v_h}^2+\norm{\mu_h}^2_h}~~~~~~~~~~~~~~~~~~~~~\text{(by \eqref{basic_d})}\\ 
            &\lesssim\frac{|\mu_h|^2_h}{(1-sCh^2)\norm{u_0}^2}~~~~~~~~~~~~~~~~~~\text{(by \eqref{33} and \eqref{44})}\\
            &\lesssim\frac{s\norm{u_0}^2}{(1-sCh^2)\norm{u_0}^2}~~~~~~~~~~~~~~~~~~\text{(by \eqref{55})}\\
            &\lesssim\frac{s}{1-sCh^2},
        \end{split}
    \end{displaymath}
    which indicates the inequality \eqref{eq_up}  immediately.   
\end{proof}

In light of Theorems \ref{bound1}- \ref{bound2}, it's straightforward to 
derive the following theorem.
\begin{thm}\label{thm_condition_num}
   Let $\mathcal A_h$ be the operator defined by \eqref{def_A}, then it holds
    \begin{equation}
        \kappa(\mathcal A_h) \lesssim h^{-2},
    \end{equation}
    where    $\kappa(\mathcal A_h):=\frac{\lam_{max}(\mathcal A_h)}{\lam_{min}(\mathcal A_h)}$, with    $\lam_{max}(\mathcal A_h)$, $\lam_{min}(\mathcal A_h)$ denoting the largest and smallest eigenvalues
    of $\mathcal A_h$ respectively.
    Further more, it holds
    \begin{equation}\label{cond-Ah}
        \kappa(\mathcal A_h)= O(h^{-2})
    \end{equation}
 if $h$ is sufficiently small.  
\end{thm}
\begin{rem}
    Let $A_h$ be  the stiffness matrix of $a_h(\cdot,\cdot)$  defined by \eqref{def_A}, then we easily have $\kappa(A_h)\sim \kappa(\mathcal A_h) = O(h^{-2})$.
\end{rem}

\section{Two-level algorithm}\label{sec_multigrid}

In this section, we analyze a two-level algorithm for the discrete system \eqref{operator system}. For the sake of clarity,
our description is in operator form. 
\newcommand{\we}{\widetilde}
\subsection{Algorithm definition} 
Set
\begin{equation}\label{tilde-Wh0}
    \we V_h := \{\we v_h\in H^1_0(\Omega):\we v_h|_T\in P_1(T),~\forall T\in\mathcal T_h\}.
\end{equation}
We first define the prolongation operator $I_h:\we V_h\to V_h\times M_h^0$ as follows:
for any $\we v_h\in\we V_h$, $I_h\we v_h: = (I_h^i\we v_h,I_h^b\we v_h)\in V_h\times M_h^0$
satisfies
\begin{displaymath}
    \left\{
        \begin{array}{rclll}
            \int_TI_h^i\we v_hv &=& \int_T\we v_hv,&\forall v\in V(T),&\forall T\in\mathcal T_h,\\
            \int_FI_h^b\we v_h\mu &=& \int_F\we v_h\mu,&\forall\mu\in M(F),&\forall F\in\mathcal F_h.
        \end{array}
    \right.
\end{displaymath}
Then define the adjoint operator, $I_h^t$, of $I_h$ by
\begin{displaymath}
    (I_h^t(v_h,\mu_h),\we v_h)_{\Omega} := ((v_h,\mu_h),I_h\we v_h)_h,
    ~\forall (v_h,\mu_h)\in V_h\times M_h^0,\forall\we v_h\in\we V_h.
\end{displaymath}
\par 
Define $\we{\mathcal A_h}:\we V_h\to\we V_h$ by
\begin{equation}\label{eq_def_we_A_h}
    (\we{\mathcal A_h}\we u_h,\we v_h)_{\Omega} := (\bm a\bmnabla\we u_h,\bmnabla\we v_h)_{\Omega},
    ~\forall\we u_h,\we v_h\in\we V_h.
\end{equation}
\begin{rem}\label{rem_I_h}
    By the definition of $I_h$, it's trivial to verify that $\bmnabla_wI_h\we v_h = \bmnabla\we v_h,
    ~\forall\we v_h\in\we V_h$. Thus we have the following important relationship:
    \begin{equation}\label{A-WeA}
        \we{\mathcal A_h} = I_h^t\mathcal A_hI_h.
    \end{equation}
\end{rem}

Let $\we{\mathcal R_h}:\we V_h\to\we V_h$ be a good approximation of $\we{\mathcal A_h}^{-1}$ and 
define 
$\we{\mathcal R_h}^t$ by
\begin{displaymath}
    (\we{\mathcal R_h}^t\we u_h,\we v_h)_{\Omega}: = (\we u_h,\we{\mathcal R_h}\we v_h)_{\Omega},
    ~\forall\we u_h,\we v_h\in\we V_h.
\end{displaymath}
\par
Let $\mathcal R_h:V_h\times M_h^0\to V_h\times M_h^0$ be a good approximation of $\mathcal A_h^{-1}$.
and let  $\mathcal R_h^t:V_h\times M_h^0\to V_h\times M_h^0$ be defined by
\begin{displaymath}
    (\mathcal R_h^t(u_h,\lam_h),(v_h,\mu_h))_h: = ((u_h,\lam_h),\mathcal R_h(v_h,\mu_h))_h,
    ~\forall (u_h,\lam_h),(v_h,\mu_h)\in V_h\times M_h^0.
\end{displaymath}
\par 
Using the above operators, we define an ingredient operator $\mathcal B_h:V_h\times M_h^0\to V_h\times M_h^0$ as follows:
\begin{algorithm}\label{algorithm 1} 
    For any $b_h\in V_h\times M_h^0$, define $\mathcal B_hb_h = (v_h^4,\mu_h^4)$ by \\
    $~~~~1$. Smooth: $(v_h^1,\mu_h^1):=\mathcal R_hb_h$,\\
    $~~~~2$. Correct: $(v_h^2,\mu_h^2):=(v_h^1,\mu_h^1)+I_h\we{\mathcal R_h}I_h^t(b_h-\mathcal A_h(v_h^1,\mu_h^1))$,\\
    $~~~~3$. Correct: $(v_h^3,\mu_h^3):=(v_h^2,\mu_h^2)+I_h\we{\mathcal R_h}^tI_h^t(b_h-\mathcal A_h(v_h^2,\mu_h^2))$,\\
    $~~~~4$. Smooth: $(v_h^4,\mu_h^4) := (v_h^3,\mu_h^3) + \mathcal R_h^t(b_h-\mathcal A_h(v_h^3,\mu_h^3))$.
\end{algorithm}
\par

We are now in a position to present the two-level algorithm
for the  system \eqref{operator system}.

\begin{algorithm}\label{algorithm2}
    Set $(u_h^0,\lam_h^0)=(0,0)$, \\
    $~~~~$for $j=1,2,\ldots$ till convergence\\
    $~~~~~~~~(u_h^j,\lam_h^j) := (u_h^{j-1},\lam_h^{j-1})+\mathcal B_h(b_h-\mathcal A_h(u_h^{j-1},\lam_h^{j-1}))$.\\
    $~~~~$end
\end{algorithm}
\subsection{Convergence analysis}
At first, we introduce some abstract notations. Let $X$ be a finite dimensional Hilbert space 
with inner product $(\cdot,\cdot)$ and its induced norm $\norm{\cdot}$. For any linear SPD operator $A:X\to X$,  the notation $(\cdot,\cdot)_A:=(A\cdot,\cdot)$ defines an inner product on 
$X$ and we denote by $\norm{\cdot}_A$  the norm induced by $(\cdot,\cdot)_A$. Let $B:X\to X$ 
be a linear operator with
\begin{displaymath}
    \norm{B}_A:=\sup_{0\ne x\in X}\frac{\norm{Bx}_A}{\norm{x}_A}.
\end{displaymath}
\par
From the definition of $\mathcal B_h$ in {\bf Algorithm \ref{algorithm 1}}, we easily obtain the following lemma.
\begin{lem}
   It holds
    \begin{equation}
        I-\mathcal B_h\mathcal A_h 
        = (I-\mathcal R_h^t\mathcal A_h)(I-I_h\widetilde{\mathcal R_h}^tI_h^t\mathcal A_h)
        (I-I_h\widetilde{\mathcal R_h}I_h^t\mathcal A_h)(I-\mathcal R_h\mathcal A_h).
    \end{equation}
\end{lem}
It's trivial to verify that $I-\mathcal B_h\mathcal A_h$ is symmetric semi-positive 
definite with respect to the inner product $(\cdot,\cdot)_{\mathcal A_h}$, and thus 
it follows $\lam_{max}(\mathcal B_h\mathcal A_h)\leqslant 1$ and
\begin{equation}\label{I-BA}
    \norm{I-\mathcal B_h\mathcal A_h}_{\mathcal A_h}=1-\lam_{min}(\mathcal B_h\mathcal A_h).
\end{equation}
Now we introduce the symmetrizations of $\mathcal R_h$ and $\we{\mathcal R_h}$, i.e.
\begin{eqnarray}
    \overline{\mathcal R_h} &:=& 
    \mathcal R_h^t+\mathcal R_h-\mathcal R_h^t\mathcal A_h\mathcal R_h,\label{def_sysm_R}\\
    \overline{\we{\mathcal R_h}} &:=& 
    \we{\mathcal R_h}^t+\we{\mathcal R_h}-\we{\mathcal R_h}^t\we{\mathcal A_h}\we{\mathcal R_h}
    \label{def_sysm_we_R},
\end{eqnarray}
and make the following assumption.
\begin{assumption}\label{assum_smoother} The operators $\mathcal R_h$ and $\we{\mathcal R_h}$ are such that
    \begin{eqnarray}
        \norm{I-\overline{\mathcal R_h}\mathcal A_h}_{\mathcal A_h} &<& 1,\label{assum_overline_R}\\
        \norm{I-\overline{\we{\mathcal R_h}}\we{\mathcal A_h}}_{\we{\mathcal A_h}} &<& 1.\label{assum_overline_we_R}
    \end{eqnarray}
\end{assumption}
\begin{rem}\label{rem_smoother}
    It follows from \eqref{assum_overline_R} that $\overline{\mathcal R_h}$ is SPD with respect
    to the inner product $(\cdot,\cdot)_h$. Then it follows from
    \begin{displaymath}
        \overline{\mathcal R_h}=\mathcal R_h^t(\mathcal R_h^{-t}+\mathcal R_h^{-1}-\mathcal A_h)
        \mathcal R_h
    \end{displaymath}
    that $\mathcal R_h^{-t}+\mathcal R_h^{-1}-\mathcal A_h$ is SPD with respect to the inner product
    $(\cdot,\cdot)_h$. Similarly, $\overline{\we{\mathcal R_h}}$ and 
    $\we{\mathcal R_h}^{-t}+\we{\mathcal R_h}^{-1}-\we{\mathcal A_h}$
    are both SPD with respect to the inner product $(\cdot,\cdot)_{\Omega}$.
\end{rem}
Following the basic idea of the X-Z identity (\cite{Xu-Z2002},\cite{XZ2},\cite{Chen_XZ}), we have
the following ingredient theorem.
\begin{thm}\label{thm_X-Z}
    Under  {\bf Assumption \ref{assum_smoother}}, $\mathcal B_h$ is a SPD operator with respect to 
    the inner product $(\cdot,\cdot)_h$, and, for any $(u_h,\lam_h)\in V_h\times M_h^0$, it holds
    \begin{equation}\label{eq_XZ}
        \begin{split}
            &(\mathcal B_h^{-1}(u_h,\lam_h),(u_h,\lam_h))_h\\
            =&\inf_{\substack{(v_h,\mu_h)+I_h\we v_h=(u_h,\lam_h)\\(v_h,\mu_h)\in V_h\times M_h^0,\we v_h\in\we V_h}}
            \norm{(v_h,\mu_h) + \mathcal R_h^t\mathcal A_hI_h\we v_h}^2_{\overline{\mathcal R_h}^{-1}}
            +\norm{\we v_h}^2_{\overline{\we{\mathcal R_h}}^{-1}}.
        \end{split}
    \end{equation}
    Further more, it holds the following extended X-Z identity:
    \begin{equation}\label{eq_convergence_xz}
        \norm{I-\mathcal B_h\mathcal A_h}_{\mathcal A_h}= 1 - \frac{1}{K},
    \end{equation}
    where
    \begin{equation}\label{def_K}
        K = \sup_{\norm{(u_h,\lam_h)}_{\mathcal A_h}=1}
        \inf_{\substack{(v_h,\mu_h)+I_h\we v_h = (u_h,\lam_h)\\(v_h,\mu_h)\in V_h\times M_h^0,\we v_h\in\we V_h}}
        \norm{(v_h,\mu_h) + \mathcal R_h^t\mathcal A_hI_h\we v_h}^2_{\overline{\mathcal R_h}^{-1}}
        +\norm{\we v_h}^2_{\overline{\we{\mathcal R_h}}^{-1}}.
    \end{equation}
\end{thm}
\begin{proof} The desired results follow from a trivial modification of the proof of the X-Z identity in \cite{Chen_XZ}.  
    For completeness we sketch the proof of this theorem. 
   We note that $\we V_h\not\subset V_h\times M_h^0$ means the corresponding spaces here are nonnested.
    
    Denote $X_h:=(V_h\times M_h^0)\times\we V_h$  and define the inner product $[\cdot,\cdot]$ 
    on $X_h$ by
    \begin{displaymath}
        [(a,b),(c,d)]:=(a,c)_h+(b,d)_{\Omega},~\forall (a,b),(c,d)\in X_h.
    \end{displaymath}
    Introduce the operator  $\Pi_h:X_h\to V_h\times M_h^0$ and its adjoint operator $\Pi_h^t:V_h\times M_h^0\to X_h$ with
    \begin{displaymath}
        \begin{array}{rcll}
            \Pi_h&:=& (I~I_h), &\text{i.e. }  \Pi_h(a,b)=a+I_hb  \text{ for any } (a,b)\in X_h,\\
            \Pi_h^t&:=&\left(\begin{array}{c}I\\I_h^t\end{array}\right), &
            \text{i.e. }  \Pi_h^ta=\left(\begin{array}{c}a\\I_h^ta\end{array}\right) \text{ for any } a\in V_h\times M_h^0.
        \end{array}
    \end{displaymath}
    Obviously, we have $(\Pi_h\tilde a,b)_h=[\tilde a,\Pi^tb],~\forall \tilde a\in X_h,\forall b\in V_h\times M_h^0$.
    \par
    Now define  
    \begin{displaymath}
        \we{\we{\mathcal A_h}}
       : =\left(
            \begin{array}{cc}
                \mathcal A_h & \mathcal A_hI_h\\
                I_h^t\mathcal A_h & I_h^t\mathcal A_hI_h
            \end{array}
        \right)
        ,~\we{\we{\mathcal B_h}}
       : =\left(
            \begin{array}{cc}
                \mathcal R_h^{-1} & 0\\
                I_h^t\mathcal A_h & \we{\mathcal R_h}^{-1}
            \end{array}
        \right)^{-1},
    \end{displaymath}
    and denote by $\we{\we{\mathcal D_h}}$  the diagonal of $\we{\we{\mathcal A_h}}$.
    \par
    For any $b_h\in V_h\times M_h^0$, set 
    \begin{displaymath}
        \begin{array}{rcl}
            w_1 &:=& \we{\we{\mathcal B_h}}\Pi_h^tb_h,\\
            w_2 &:=& w_1+\we{\we{\mathcal B_h}}^t(\Pi_h^tb_h-\we{\we{\mathcal A_h}}w_1).
        \end{array}
    \end{displaymath}
    Then it holds $\Pi_hw_2 = \Pi_h\overline{\we{\we{\mathcal B_h}}}\Pi_h^tb_h$, where
    \begin{displaymath}
        \overline{\we{\we{\mathcal B_h}}}  := \we{\we{\mathcal B_h}}^t+\we{\we{\mathcal B_h}}
        -\we{\we{\mathcal B_h}}^t\we{\we{\mathcal A_h}}\we{\we{\mathcal B_h}}.
    \end{displaymath}
    It's easy to verify that $\Pi_hw_2=\mathcal B_hb_h$, which yields
    \begin{equation}\label{B-WeB}
        \mathcal B_h = \Pi_h\overline{\we{\we{\mathcal B_h}}}\Pi_h^t.
    \end{equation}
    Denoting $\we{\we{\mathcal R_h}}: = \text{diag}(\mathcal R_h,\we{\mathcal R_h})$, we have
    \begin{displaymath}
        \begin{split}
            \overline{\we{\we{\mathcal B_h}}}
            &=\we{\we{\mathcal B_h}}^t(\we{\we{\mathcal B_h}}^{-t}+\we{\we{\mathcal B_h}}^{-1}
            -\we{\we{\mathcal A_h}})\we{\we{\mathcal B_h}}\\
            &=\we{\we{\mathcal B_h}}^t(\we{\we{\mathcal R_h}}^{-t}+\we{\we{\mathcal R_h}}^{-1}-\we{\we{\mathcal D_h}})
            \we{\we{\mathcal B_h}}\\
            &=\we{\we{\mathcal B_h}}^t\we{\we{\mathcal R_h}}^{-t}
            \overline{\we{\we{\mathcal R_h}}}\we{\we{\mathcal R_h}}^{-1}
            \we{\we{\mathcal B_h}},
        \end{split}
    \end{displaymath}
    where $\overline{\we{\we{\mathcal R_h}}} = \we{\we{\mathcal R_h}}^t+\we{\we{\mathcal R_h}}
    -\we{\we{\mathcal R_h}}^t\we{\we{\mathcal D_h}}\we{\we{\mathcal R_h}}$. By \eqref{A-WeA} 
    we also have $\overline{\we{\we{\mathcal R_h}}} =diag(\overline{\mathcal R_h},
    \overline{\we{\mathcal R_h}})$.
    From Remark \ref{rem_smoother}, it follows that 
    $\overline{\we{\we{\mathcal R_h}}}$
    is SPD with respect to $[\cdot,\cdot]$.
    Thus $\overline{\we{\we{\mathcal B_h}}}$ is SPD with respect to $[\cdot,\cdot]$. Then from 
    {\bf Theorem 1} in \cite{Chen_XZ} and \eqref{B-WeB} it follows 
    \begin{equation}\label{3331}
        (\mathcal B_h^{-1}(u_h,\lam_h),(u_h,\lam_h))_h
        =\inf_{\substack{\Pi_hw_h=(u_h,\lam_h)\\w_h\in X_h}}
        [\overline{\we{\we{\mathcal B_h}}}^{-1}w_h,w_h] .
    \end{equation} 
    In view of 
    \begin{displaymath}
        \begin{split}
            \overline{\we{\we{\mathcal B_h}}}^{-1}
            =\we{\we{\mathcal B_h}}^{-1}\we{\we{\mathcal R_h}}
            \overline{\we{\we{\mathcal R_h}}}^{-1}\we{\we{\mathcal R_h}}^t\we{\we{\mathcal B_h}}^{-t}
            =\left(\begin{array}{cc}I & 0\\ I_h^t\mathcal A_h\mathcal R_h& I\end{array}\right)
            \overline{\we{\we{\mathcal R_h}}}^{-1}
            \left(\begin{array}{cc}I& \mathcal R_h^t\mathcal A_hI_h\\0 & I\end{array}\right),
        \end{split}
    \end{displaymath}
   the identity \eqref{eq_XZ} follows immediately from \eqref{3331}. The extended X-Z identity \eqref{eq_convergence_xz}  is
    just a trivial conclusion from \eqref{eq_XZ}.
\end{proof}
We define the operator $P_h:M_h^0\to\we V_h$ as follows. For any $\lam_h\in M_h^0$, $P_h\lam_h$ satisfies
\begin{displaymath}
    \left\{
        \begin{array}{rcll}
            P_h\lam_h(\bm x) &=& \sum\limits_{T\in\omega_{\bm x}}\frac{\sum\limits_{T\in\omega_{\bm x}}m_T(\lam_h)}
            {\sum\limits_{T\in\omega_{\bm x}}1},&\text{for each interior vertex $\bm x$ of $\mathcal T_h$,}\\
            P_h\lam_h(\bm x) &=& 0, &\text{for each vertex $\bm x\in\partial\Omega$,}
        \end{array}
    \right.
\end{displaymath}
where the set $\omega_{\bm x}:=\{T\in\mathcal T_h:\bm x \text{ is a vertex of $T$}\}$.
\par
As for the operator $P_h$, we have the following important estimates.
\begin{lem}
    For any $(u_h,\lam_h)\in V_h\times M_h^0$, it holds
    \begin{eqnarray}
        \norm{(I-I_h^bP_h)\lam_h}_h &\lesssim& h\norm{(u_h,\lam_h)}_{\mathcal A_h},\label{eq_esti_I_b}\\
        \norm{u_h-I_h^iP_h\lam_h} &\lesssim& h\norm{(u_h,\lam_h)}_{\mathcal A_h},\label{eq_esti_I_i}
    \end{eqnarray}
   which further indicate
    \begin{equation}\label{eq_esti_I_h}
        \norm{(u_h,\lam_h)-I_hP_h\lam_h}_h\lesssim h\norm{(u_h,\lam_h)}_{\mathcal A_h}.
    \end{equation}
\end{lem}
\begin{proof}
    We denote by $\omega_T$   the set $\{T'\in\mathcal T_h:\text{ $T'$ and $T$ share a vertex}\}$ and by
    $\mathcal N(T)$ the set of all vertexes of $T$.
    Since
    \begin{equation}\label{3456}
        \begin{split}
            &h_T\norm{I_h^bP_h\lam_h-m_T(\lam_h)}^2_{\partial T}\\
            \leqslant&~ h_T\norm{P_h\lam_h-m_T(\lam_h)}^2_{\partial T}\\
            \lesssim&~ h_T^d\sum_{\bm x\in\mathcal N(T)}|P_h\lam_h(\bm x)-m_T(\lam_h)|^2\\
            \lesssim&~ h_T^d\sum_{\bm x\in\mathcal N(T)}
            \sum_{\substack{T_1,T_2\in\omega_{\bm x}\\\text{$T_1$ and $T_2$ share a same face}}}
            |m_{T_1}(\lam_h)-m_{T_2}(\lam_h)|^2\\
            \lesssim&~ h_T^2\sum_{T'\in\omega_T}\vertiii{\lam_h}^2_{h,\partial T'},
        \end{split}
    \end{equation}
    we have
    \begin{displaymath}
        \begin{split}
            h_T\norm{(I-I_h^bP_h)\lam_h}^2_{\partial T}
            &\lesssim h_T\norm{\lam_h-m_T(\lam_h)}^2_{\partial T} + h_T\norm{I_h^bP_h\lam_h-m_T(\lam_h)}^2_{\partial T}\\
            &\lesssim h_T^2\sum_{T'\in\omega_T}\vertiii{\lam_h}^2_{h,\partial T'}.
        \end{split}
    \end{displaymath}
    Then the estimate \eqref{eq_esti_I_b} follows immediately from \eqref{basic_d}.\par
    
    On the other hand, since
    \begin{displaymath}
        \begin{split}
            \norm{I_h^iP_h\lam_h-m_T(\lam_h)}^2_T
            &\leqslant\norm{P_h\lam_h-m_T(\lam_h)}^2_T\\
            &\lesssim h_T\norm{P_h\lam_h-m_T(\lam_h)}^2_{\partial T}\\
            &\lesssim h_T^2\sum_{T'\in\omega_T}\vertiii{\lam_h}^2_{h,\partial T'},
            ~~~~~~~~~~~~~\text{(by \eqref{3456})}
        \end{split}
    \end{displaymath}
    it holds
    \begin{displaymath}
        \begin{split}
            \norm{u_h-I_h^iP_h\lam_h}^2_T
            &\lesssim\norm{u_h-m_T(\lam_h)}^2_T + \norm{I_h^iP_h\lam_h-m_T(\lam_h)}^2_T\\
            &\lesssim\norm{u_h-m_T(\lam_h)}^2_T + \sum_{T'\in\omega_T}h_T^2\vertiii{\lam_h}^2_{h,\partial T'}.
        \end{split}
    \end{displaymath}
    Then the estimate \eqref{eq_esti_I_i} also follows immediately from \eqref{basic_d}.\par
    
    Finally, the result \eqref{eq_esti_I_h} is a trivial conclusion from \eqref{eq_esti_I_b} and \eqref{eq_esti_I_i}.
\end{proof}
\begin{lem}
    For any $(u_h,\lam_h)\in V_h\times M_h^0$, it holds
    \begin{equation}\label{eq_IP}
        \norm{I_hP_h\lam_h}_{\mathcal A_h}=\norm{P_h\lam_h}_{\we{\mathcal A_h}}
        \lesssim\norm{(u_h,\lam_h)}_{\mathcal A_h}.
    \end{equation}
\end{lem}
\begin{proof}
    The relation $\norm{I_hP_h\lam_h}_{\mathcal A_h}=\norm{P_h\lam_h}_{\we{\mathcal A_h}}$ follows from \eqref{A-WeA}.  It suffices 
    to prove the inequality of \eqref{eq_IP}.  Since
    \begin{displaymath}
        \begin{split}
            |P_h\lam_h|^2_{1,T}
            &=|P_h\lam_h-m_T(\lam_h)|^2_{1,T}\\
            &\lesssim h_T^{-2}\norm{P_h\lam_h-m_T(\lam_h)}^2_T~~~~~~~~~~~\text{(by inverse estimate)}\\
            &\lesssim h_T^{-1}\norm{P_h\lam_h-m_T(\lam_h)}^2_{\partial T}\\
            &\lesssim\sum_{T'\in\omega_T}\vertiii{\lam_h}^2_{h,\partial T'},
            ~~~~~~~~~~~~~~~~~~~~~~~\text{(by \eqref{3456})}
        \end{split}
    \end{displaymath}
    we have
    \begin{displaymath}
        \norm{P_h\lam_h}_{\we{\mathcal A_h}}\sim|P_h\lam_h|_{1,\Omega}\lesssim\vertiii{\lam_h}_h,
    \end{displaymath}
    which, together with \eqref{basic_d}, implies   the desired conclusion.
\end{proof}
\begin{assumption}\label{assum_R_h}
    The smoother $\mathcal R_h:V_h\times M_h^0\to V_h\times M_h^0$ is SPD with respect ro $(\cdot,\cdot)_h$ and satisfies
    \begin{equation}
        \sigma(\mathcal R_h\mathcal A_h)\subset (0, 1],
    \end{equation}
    where $\sigma(\mathcal R_h\mathcal A_h)$ denotes the set of all eigenvalues of $\mathcal R_h\mathcal A_h$. 
    What's more, for any $(u_h,\lam_h)\in V_h\times M_h^0$, it holds
    \begin{equation}
        \norm{(u_h,\lam_h)}^2_{\overline{\mathcal R_h}^{-1}}
        \leqslant C_R\lam_{max}(\mathcal A_h)\norm{(u_h,\lam_h)}^2_h,
    \end{equation}
    where  $\overline{\mathcal R_h}$ is the symmetrization of $ \mathcal R_h$,  and $C_R$ denotes a positive constant.
\end{assumption}
\begin{rem}
    If we take $\mathcal R_h=\frac{1}{\lam_{max}(\mathcal A_h)}I$, then it holds
    $\overline{\mathcal R_h}^{-1}=\lam^2_{max}(\mathcal A_h)(2\lam_{max}(\mathcal A_h)I-\mathcal A_h)^{-1}$.
   In this case it is obvious that $C_R=1$. If we take $\mathcal R_h$ to be the symmetric Gauss-Seidel smoother,
    then $C_R$ is a bounded positive constant   independent of the mesh size $h$.
\end{rem}
\begin{rem}
    Suppose {\bf Assumption \ref{assum_R_h}} holds, then the relation  $I-\overline{\mathcal R_h}\mathcal A_h
    =(I-\mathcal R_h\mathcal A_h)^2$ leads to $\sigma(I-\overline{\mathcal R_h}\mathcal A_h)
    \subset [0,1)$ and   it follows $\norm{I-\overline{\mathcal R_h}\mathcal A_h}_{\mathcal A_h}<1$.
\end{rem}
\begin{lem}
    Under {\bf Assumption \ref{assum_R_h}}, for any $(u_h,\lam_h)\in V_h\times M_h^0$,
    it holds
    \begin{equation}\label{eq_esti_S_h}
        \norm{\mathcal R_h\mathcal A_h(u_h,\lam_h)}_{\overline{\mathcal R_h}^{-1}}
        \leqslant\norm{(u_h,\lam_h)}_{\mathcal A_h}.
    \end{equation}
\end{lem}
\begin{proof}
    Denoting $\mathcal S_h:= \mathcal R_h\mathcal A_h$ and thanks to 
    \begin{displaymath}
        \overline{\mathcal R_h} = 2\mathcal R_h-\mathcal R_h\mathcal A_h\mathcal R_h
        =(2\mathcal S_h-\mathcal S_h^2)\mathcal A_h^{-1},
    \end{displaymath}
    we have
    \begin{equation}\label{9991}
        \begin{split}
            \norm{\mathcal R_h\mathcal A_h(u_h,\lam_h)}^2_{\overline{\mathcal R_h}^{-1}}
            &=(\overline{\mathcal R_h}^{-1}\mathcal R_h\mathcal A_h(u_h,\lam_h),\mathcal R_h\mathcal A_h(u_h,\lam_h))_h\\
            &=(\mathcal A_h(2\mathcal S_h-\mathcal S_h^2)^{-1}\mathcal S_h(u_h,\lam_h),
            \mathcal R_h\mathcal A_h(u_h,\lam_h))_h\\
            &=(\mathcal S_h(2\mathcal S_h-\mathcal S_h^2)^{-1}\mathcal S_h(u_h,\lam_h),(u_h,\lam_h))_{\mathcal A_h},
        \end{split}
    \end{equation}
  which,   together with the fact that  $\mathcal S_h$ is SPD with respect to $(\cdot,\cdot)_{\mathcal A_h}$ 
    and the inequality
    \begin{displaymath}
        t(2t-t^2)^{-1}t\leqslant 1,t\in (0,1],
    \end{displaymath}
  yields the desired estimate \eqref{eq_esti_S_h}.
\end{proof}
Finally, we state the following convergence  theorem.
\begin{thm}\label{thm_two_level_convergence}
    Under   {\bf Assumptions \ref{assum_smoother}-\ref{assum_R_h}}, it holds
    \begin{equation}\label{convergence-two}
        \norm{(I-\mathcal B_h\mathcal A_h)}_{\mathcal A_h}\leqslant 1-\frac{1}{K},
    \end{equation}
    where 
    \begin{equation}
        K\lesssim \left(1+C_R+\frac{1}{1-\norm{I-\overline{\we{\mathcal R_h}}\we{\mathcal A_h}}_{\we{\mathcal A_h}}}\right).
    \end{equation}
\end{thm}
\begin{proof}
    For any $(u_h,\lam_h)\in V_h\times M_h^0$, set  
    $$\we v_h:=P_h\lam_h,\quad  (v_h,\mu_h) := (u_h,\lam_h)-I_h\we v_h,$$ we then obtain
    \begin{displaymath}
        \begin{split}
            \norm{(v_h,\mu_h)+\mathcal R_h\mathcal A_hI_h\we v_h}^2_{\overline{\mathcal R_h}^{-1}}
            &\lesssim \norm{(v_h,\mu_h)}^2_{\overline{\mathcal R_h}^{-1}}
            +\norm{\mathcal R_h\mathcal A_hI_h\we v_h}^2_{\overline{\mathcal R_h}^{-1}}\\
            &\lesssim\norm{(v_h,\mu_h)}^2_{\overline{\mathcal R_h}^{-1}}
            +\norm{I_h\we v_h}^2_{\mathcal A_h}~~~~~~~~~~~~~~~~~~~~\text{(by \eqref{eq_esti_S_h})}\\
            &\lesssim\norm{(v_h,\mu_h)}^2_{\overline{\mathcal R_h}^{-1}}
            +\norm{(u_h,\lam_h)}^2_{\mathcal A_h}~~~~~~~~~~~~~~~~~\text{(by \eqref{eq_IP})}\\
            &\lesssim C_R\lam_{max}(\mathcal A_h)\norm{(v_h,\mu_h)}_h^2
            +\norm{(u_h,\lam_h)}^2_{\mathcal A_h}~~~~~\text{(by {\bf Assumption \ref{assum_R_h}})}\\
            &\lesssim (1+C_R)\norm{(u_h,\lam_h)}^2_{\mathcal A_h},
        \end{split}
    \end{displaymath}
  where,   in the last inequality, we have used the estimate \eqref{eq_esti_I_h} and the fact   $\lam_{max}(\mathcal A_h)\sim h^{-2}$ derived from 
    Theorems \ref{bound1}-\ref{bound2}.
    Similar to \eqref{I-BA}, we have
    \begin{displaymath}
        \norm{I-\overline{\we{\mathcal R_h}}\we{\mathcal A_h}}_{\we{\mathcal A_h}} 
        = 1 - \lam_{min}(\overline{\we{\mathcal R_h}}\we{\mathcal A_h}),
    \end{displaymath}
    and it follows
    \begin{displaymath}
        \begin{split}
            \norm{\we v_h}^2_{\overline{\we{\mathcal R_h}}^{-1}}
            &\leqslant\frac{1}{\lam_{min}(\overline{\we{\mathcal R_h}}\we{\mathcal A_h})}
            \norm{\we v_h}^2_{\we{\mathcal A_h}}
            =\frac{1}{1-\norm{I-\overline{\we{\mathcal R_h}}\we{\mathcal A_h}}_{\we{\mathcal A_h}}}
            \norm{\we v_h}^2_{\we{\mathcal A_h}}\\
            &\lesssim\frac{1}{1-\norm{I-\overline{\we{\mathcal R_h}}\we{\mathcal A_h}}_{\we{\mathcal A_h}}}
            \norm{(u_h,\lam_h)}^2_{\mathcal A_h}.~~~~~~~~~~~~~~\text{(by \eqref{eq_IP})}\\
        \end{split}
    \end{displaymath}
    Therefore, we have
    \begin{displaymath}
        \begin{split}
            &\norm{(v_h,\mu_h)+\mathcal R_h\mathcal A_hI_h\we v_h}^2_{\overline{\mathcal R_h}^{-1}}
            +\norm{\we v_h}^2_{\overline{\we{\mathcal R_h}}^{-1}}\\
            \lesssim&(1+C_R+\frac{1}{1-\norm{I-\overline{\we{\mathcal R_h}}\we{\mathcal A_h}}_{\we{\mathcal A_h}}})
            \norm{(u_h,\lam_h)}^2_{\mathcal A_h},
        \end{split}
    \end{displaymath}
    which implies
    \begin{displaymath}
        \begin{split}
            &\sup_{\norm{(u_h,\lam_h)}_{\mathcal A_h}=1}\inf_{\substack{(v_h,\mu_h)+I_h\we v_h = (u_h,\lam_h)\\
                    (v_h,\mu_h)\in V_h\times M_h^0,\we v_h\in\we V_h}}
            \norm{(v_h,\mu_h)+\mathcal R_h^t\mathcal A_hI_h\we v_h}^2_{\overline{\mathcal R_h}^{-1}}
            +\norm{\we v_h}^2_{\overline{\we{\mathcal R_h}}^{-1}}\\
            \lesssim &(1+C_R+\frac{1}{1-\norm{I-\overline{\we{\mathcal R_h}}\we{\mathcal A_h}}_{\we{\mathcal A_h}}}).
        \end{split}
    \end{displaymath}
    As a result, the desired estimate \eqref{convergence-two} follows   from the extended X-Z identity \eqref{eq_convergence_xz}  
    in Theorem \ref{thm_X-Z}.
\end{proof}
\begin{rem}
    In our analysis, we do not use any regularity assumption of the model problem \eqref{eq_problem}.
    Thus our theory applies to the case that \eqref{eq_problem} doesn't have full elliptic
    regularity. However, if $\we{\mathcal R_h}$ is construted by standard multigrid methods, as shown 
    in \cite{Bpj3.}-\cite{Bpj1}, the lack of full regualrity will affect the convergence rate 
    $
    \norm{I-\overline{\we{\mathcal R_h}}\we{\mathcal A_h}}_{\we{\mathcal A_h}}.
    $ 
\end{rem}
\section{Numerical experiments}\label{sec_numerical}
This section reports some numerical results in two space dimensions to verify our theoretical 
results. For the model problem \eqref{eq_problem}, we set $\bm a\in\mathbb R^{2\times 2}$ to  
be the identity matrix, $\Omega=(0,1)\times(0,1)$ and we shall use the {\bf Type 2} WG method 
($k=1$). When given a coarse triangulation $\mathcal T_0$, we produce a sequence of uniformly refined
triangulations $\{\mathcal T_i:i=0,1,\ldots,5\}$ (cf. Figure \ref{mesh0_mesh1} for $\mathcal T_0$ and
$\mathcal T_1$) by a simple procedure: $\mathcal T_{j+1}$ is obtained by
connecting the midpoints of all edges of $\mathcal T_j$ for $j=0,1,2,3,4$.
\par
\begin{figure}[H]
    \begin{center}
        \includegraphics[width=0.4\linewidth]{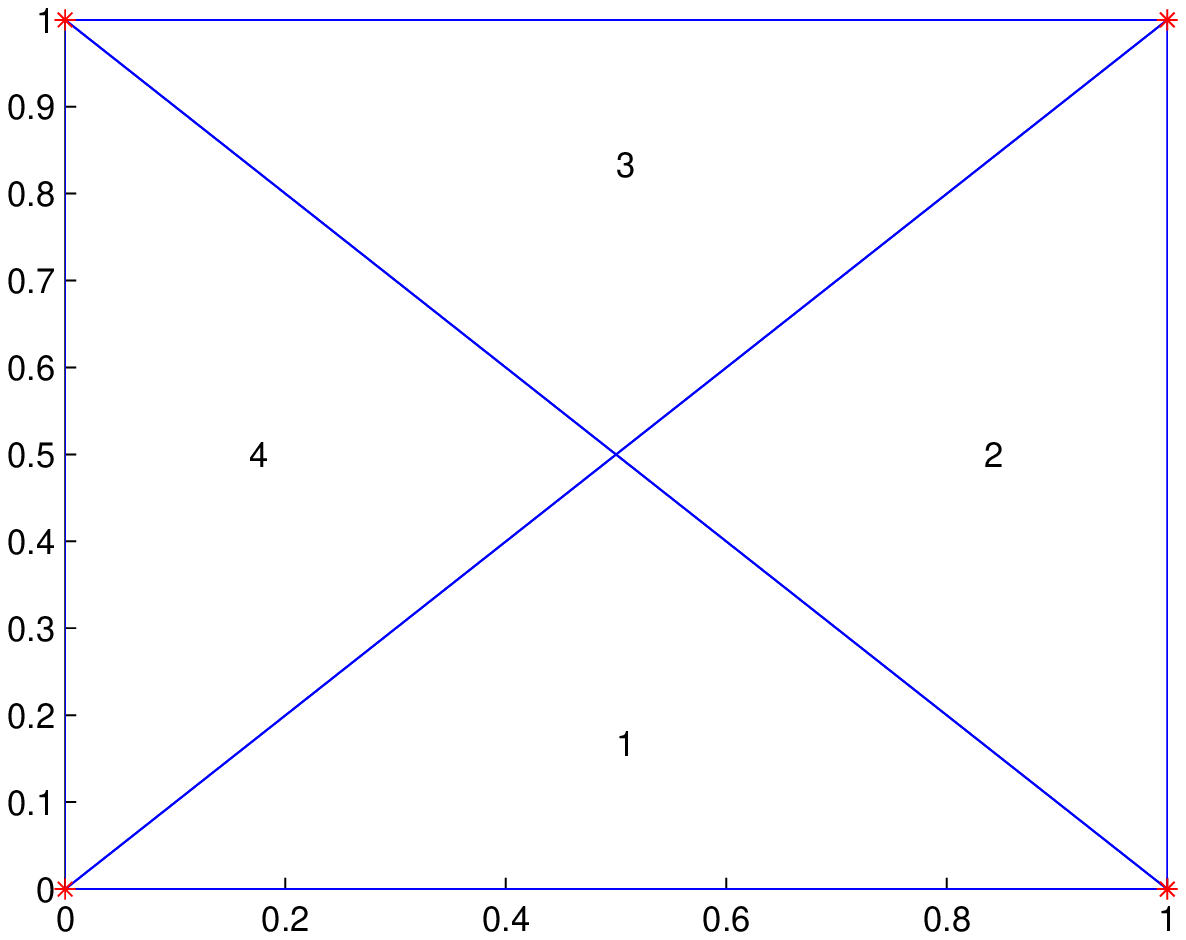}
        \includegraphics[width=0.4\linewidth]{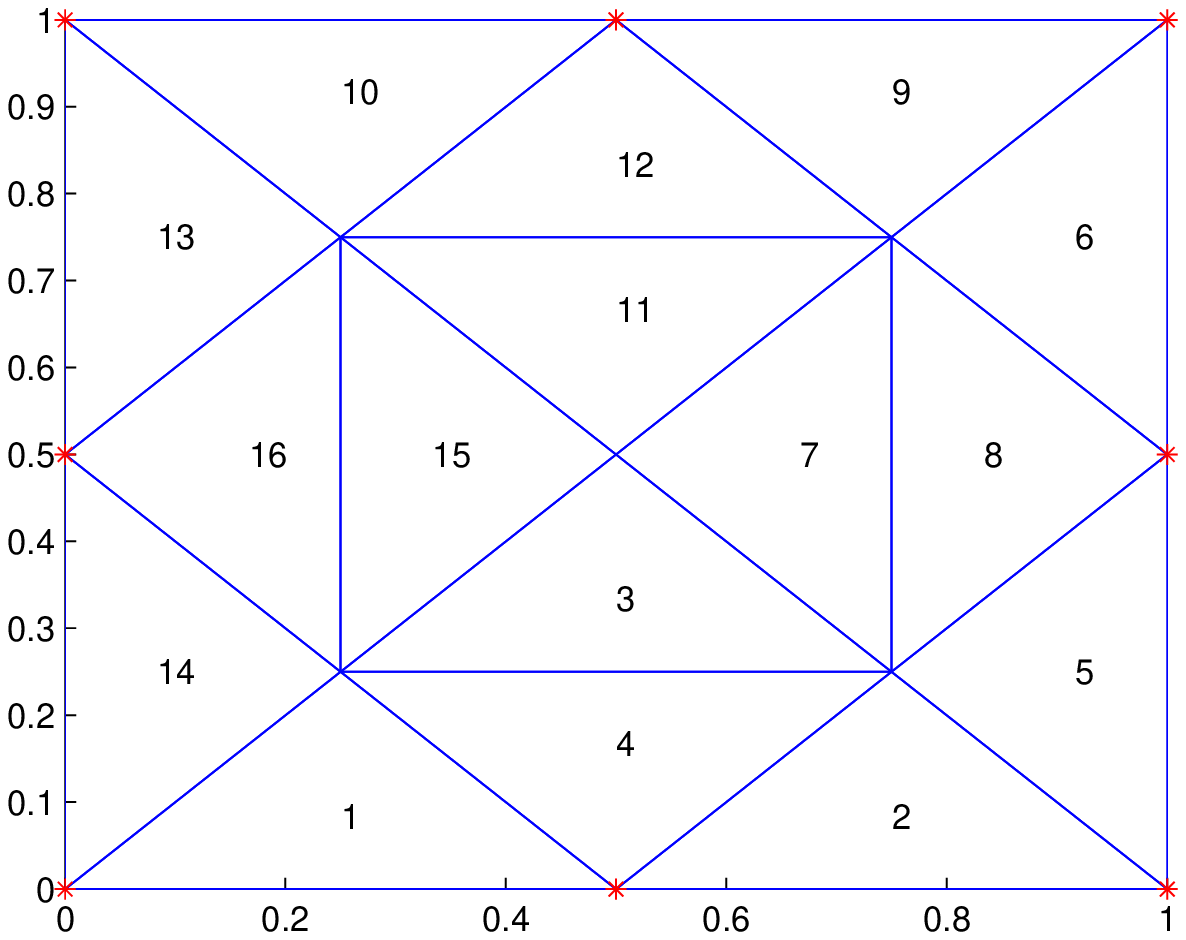}
    \end{center}
    \caption{$\mathcal T_0$ (left) and $\mathcal T_1$ (right)}\label{mesh0_mesh1}
\end{figure}
In our first experiment, we compute the smallest eigenvalue $\lam_{min}(A_h)$, the largest eigenvalue $\lam_{max}(\mathcal A_h)$ and the condition number $\kappa( A_h)$ of  the stiffness matrix  $A_h$ on each 
triangulation $\mathcal T_i$ and list them   in  Table \ref{table_cond_num}. The results  imply $\kappa(\mathcal A_h)\sim \kappa( A_h)=O(h^{-2})$, which is conformable to  Theorem \ref{thm_condition_num}.
\begin{table}[H]
    \begin{center}
        \caption{Condition numbers of $A_h$ at different triangulations}\label{table_cond_num}
        \begin{tabular}{|c|c|c|c|c|c|c|}
            \hline
            & $\mathcal T_0$ & $\mathcal T_1$ & $\mathcal T_2$ & $\mathcal T_3$ & $\mathcal T_4$ & $\mathcal T_5$ \\
            \hline
            $\lam_{min}(A_h)$ & 0.55 & 0.28 & 0.075 & 0.019 & 0.0048 & 0.0012\\
            \hline
            $\lam_{max}(  A_h)$ & 27.63 & 33.31 & 33.46 & 33.47 & 33.47 & 33.47\\
            \hline
            $\kappa( A_h)$ & 50.1 & 121.1 & 444.8 & 1746.7 & 6954.6 & 27792\\
            \hline
        \end{tabular}
    \end{center}
\end{table}
\par
In our second experiment, for each triangulation $\mathcal T_j$, we set $\mathcal T_h=\mathcal T_j$ and 
take $\mathcal R_h$ to be the $m$-times symmetric Gauss-Seidel iteration with 
$\we{\mathcal R_h} = \we{\mathcal A_h}^{-1}$. We are to solve the problem
$A_hx=b$, where $b$ is a zero vector, In order to verify the convergence, in Algorithm 
\ref{algorithm2}, we take $x_0=(1,1,\ldots,1)^t$ as the initial value, rather than the zero
vector. We stop the two-level algorithm when the initial error, i.e. $\sqrt{x_0^tA_hx_0}$, 
is reduced by a factor of $10^{-8}$. The corresponding results listed in   Table \ref{table_SGS} show that the  two-level algorithm is efficient.
\par
Our third experiment is a modification of the second one. In this experiment,   the operator
$\we{\mathcal R_h}$ is constructed by using the standard $V$-cycle multigrid method  based on the nested triangulations 
$\mathcal T_0,\mathcal T_1,\ldots,\mathcal T_j$, rather than by simply setting $\we{\mathcal R_h}=\we{\mathcal A_h}^{-1}$. 
Here we set all smoothers encountered to be the $m$-times symmetric Gauss-Seidel iterations. 
This is a practical multi-level algorithm. The numerical results  listed
  in  Tables \ref{table_multi_level}-\ref{table_multi_level_ave} show that the  multi-level algorithm is efficient.

\begin{table}[H]
    \begin{center}
        \caption{Numbers of iterations for two-level algorithm}
        \label{table_SGS}
        \begin{tabular}{|c|c|c|c|c|c|c|}
            \hline
            & $\mathcal T_0$ & $\mathcal T_1$ & $\mathcal T_2$ &  $\mathcal T_3$ & $\mathcal T_4$ & $\mathcal T_5$ \\
            \hline
            $m=1$ & 13 & 23 & 28 & 31 & 31 & 31\\
            \hline
            $m=2$ & 8 & 12 & 14 & 17 & 17 & 17\\
            \hline  
            $m=3$ & 7 & 9 & 10 & 12 & 13 & 13\\
            \hline
            $m=4$ & 6  & 8 & 9 & 9 & 10 & 10 \\
            \hline
            $m=10$ & 4 & 6  & 6  & 7  & 7 & 7 \\
            \hline
        \end{tabular}
    \end{center}
\end{table}
\begin{table}[H]
    \begin{center}
        \caption{Number of iterations for multi-level algorithm}\label{table_multi_level}
        \begin{tabular}{|c|c|c|c|c|c|}
            \hline
            & $\mathcal T_1$ & $\mathcal T_2$ & $\mathcal T_3$ & $\mathcal T_4$ & $\mathcal T_5$\\
            \hline
            $m=1$  & 22 & 28 & 31 & 31  & 31\\
            \hline
            $m=2$  & 12 & 14 & 17 & 17 & 17\\ 
            \hline
            $m=3$  &  9 & 10 & 12 & 13 & 13\\
            \hline  
        \end{tabular}
    \end{center}
\end{table}
\begin{table}[H]
    \begin{center}
        \caption{Average error reduction rates for multi-level algorithm}\label{table_multi_level_ave}
        \begin{tabular}{|c|c|c|c|c|c|}
            \hline
            & $\mathcal T_1$ & $\mathcal T_2$ & $\mathcal T_3$ & $\mathcal T_4$ & $\mathcal T_5$ \\
            \hline
            $m=1$  & 0.53 & 0.56 & 0.56 & 0.56 & 0.56\\
            \hline
            $m=2$  & 0.31 & 0.36 & 0.36 & 0.36 & 0.36\\ 
            \hline
            $m=3$  & 0.19 & 0.25 & 0.26 & 0.26 & 0.27 \\
            \hline  
        \end{tabular}
    \end{center}
\end{table}
\appendix
\section{Appendix: Proof of Lemma \ref{ref_lem_basic_inequality}}\label{append_2}
For any simplex $T$ with vertexes $\bm a_1,\bm a_2,\ldots,\bm a_{d+1}$, let $\lam_i$ be the 
barycentric coordinate function associated with the vertex $\bm a_i$ for $i=1,2,\ldots,d+1$. 
We first introduce
\begin{displaymath}
    \Lambda(T):=Q_1(T)+Q_2(T)+\ldots Q_{d+1}(T),
\end{displaymath}
where
\begin{displaymath}
    Q_i(T)=\left(\prod_{j\neq i}\lam_j\right)\text{span}
    \left\{\prod_j{\lam_j^{\alpha_j}}:\sum_{j}\alpha_j=k,\alpha_i=0\right\}, i=1,2,\ldots,d+1.
\end{displaymath}
Then we define the operator $\mathcal S:L^2(\partial T)\to\Lambda(T)$ as follows: For any 
$\mu\in L^2(\partial T)$, $\mathcal S\mu$ satisfies
\begin{displaymath}
    \int_F\mathcal S\mu q = \int_F\mu q,~\forall q\in P_k(F),\text{ for each face $F$ of $T$}.
\end{displaymath}
Finally, we define $\mathcal R:L^2(\partial T)\to P_1(T)+\Lambda(T)$ by
\begin{displaymath}
    \mathcal R\mu := \Pi^{CR}\mu+\mathcal S(\mu-\Pi^{CR}\mu).
\end{displaymath}
where $\Pi^{CR}\mu\in P_1(T)$ satisfies 
\begin{displaymath}
    \int_F\Pi^{CR}\mu := \int_F\mu,\text{ for each face $F$ of $T$}.
\end{displaymath}

By recalling  $M(\partial T):=\{\mu\in L^2(\partial T):\mu|_F\in M(F), 
\text{ for each face $F$ of $T$}\}$ and using  standard scaling arguments, it is easy
to derive the following lemma.
\begin{lem}
    For any $\mu\in M(\partial T)$, it holds
    \begin{eqnarray}
        \norm{\mu}_{h,\partial T}&\sim&\norm{\mathcal R\mu}_T,\label{65322}\\
        \vertiii{\mu}_{h,\partial T}&\sim&|\mathcal R\mu|_{1,T}\label{65323}.
    \end{eqnarray}
\end{lem}
For any $\mu_h\in M_h^0$, it is obvious that $\mathcal R\mu_h$ satisfies the $0$-th order   weak 
continuity, i.e., $\mathcal R\mu_h$ is continuous at the gravity point of each interior face 
of $\mathcal T_h$. In addition, it holds $\mathcal R\mu_h|_{\partial\Omega} = 0$. Therefore, from 
discrete Poincar\'e-Friedrichs inequalities (\cite{Brenner2003}) we have
\begin{displaymath}
    \norm{\mathcal R\mu_h}\lesssim(\sum_{T\in\mathcal T_h}|\mathcal R\mu_h|^2_{1,T})^{\frac{1}{2}}.
\end{displaymath}
Then it follows 
\begin{displaymath}
    \begin{split}
        \norm{\mu_h}^2_h
        &=\sum_{T\in\mathcal T_h}\norm{\mu_h}^2_{h,\partial T}\\
        &\sim\sum_{T\in\mathcal T_h}\norm{\mathcal R\mu_h}^2_T~~~~~~~~~~\text{(by \eqref{65322})}\\
        &\lesssim\sum_{T\in\mathcal T_h}|\mathcal R\mu_h|^2_{1,T}\\
        &\lesssim\vertiii{\mu_h}^2_h.~~~~~~~~~~~~~~~~~~~~\text{(by \eqref{65323})}\\
    \end{split}
\end{displaymath}
\section{Appendix: Proof of Lemma \ref{lem_basic_inequalities}}\label{append_1}
Denote by $\widehat T$ the referential unit simplex. For any simplex $T$, there exists an 
invertible affine map $F:\widehat T\to T$ with $F(\hat x)=A\hat x+b$ for $\hat x\in\widehat T$, 
$A\in\mathbb R^{d\times d}$ a nonsingular matrix and $b\in\mathbb R^d$. For any $p\in [L^2(T)]^s
(s=1,2,3)$ and $\mu\in L^2(\partial T)$, we understand $\widehat p$ and $\widehat\mu$ by
\begin{eqnarray}
    \widehat p(\widehat x) &=& p(x),\\
    \widehat\mu(\widehat x) &=& \mu(x),
\end{eqnarray}
where $x = F(\hat x)$ for $\hat x\in  \widehat T$.
\par
We state two well-known results as follows \cite{Ciarlet1978}: 
\begin{eqnarray}
    \norm{A} &\sim& h_T,\label{A}\\
    \norm{A^{-1}} &\sim& h_T^{-1},\label{invA}
\end{eqnarray}
where the matrix norm $\norm{\cdot}:\mathbb R^{d\times d}\to\mathbb R$ is defined by
\begin{equation}
    \norm{A}=\max_{0\neq x\in\mathbb R^d}\frac{\norm{Ax}}{\norm{x}},~\forall A \in\mathbb R^{d\times d}.
\end{equation}
Based on the above two results, it's straightforward to obtain
\begin{equation}\label{invATx}
    \norm{A^{-T}x} \sim h_T^{-1} \norm{x},~\forall x\in\mathbb R^d.
\end{equation}
\par
Using the same techniques as in the proof the properties of the famous Piola transformation (\cite{LNIM}), 
we easily obtain the lemma below.
\begin{lem}\label{transformation}
    For any $(v,\mu)\in L^2(T)\times L^2(\partial T)$, it holds
    \begin{eqnarray}
        \widehat\bmnabla^b_w\widehat\mu &=& A^{T}\widehat{\bmnabla^b_w\mu},\label{transfer_1}\\
        \widehat\bmnabla^i_w\widehat v &=& A^{T}\widehat{\bmnabla^i_w v},\label{transfer_2}\\
        \widehat\bmnabla_w(\widehat v, \widehat \mu) &=& A^{T}\widehat{\bmnabla_w(v,\mu)}.\label{transfer_3}
    \end{eqnarray}
\end{lem}

\begin{lem}\label{append_lbj}
    For any simplex T, there exist two positive constants $c_T$ and $C_T$, which only depend 
    on T and $k$, such that
    \begin{equation}
        c_T\norm{\mu}_{\partial T} \leqslant \norm{\bmnabla^b_w\mu}_T \leqslant C_T\norm{\mu}_{\partial T},~\forall
        \mu\in M(\partial T).
    \end{equation}
\end{lem}
\begin{proof}
    Assuming $\bmnabla^b_w \mu = 0$, by the definition of $\bmnabla^b_w$,  i.e. \eqref{grad-mu} we have
    \begin{displaymath}
        \langle \mu, \bm q \cdot \bm n\rangle_{\partial T} = 0,~\forall \bm q \in\bm W(T),
    \end{displaymath}
    which implies $\mu = 0$. This means the semi-norm $\norm{\bmnabla^b_w \cdot}_T$ is a norm on $M(\partial T)$. 
    Since different norms on a finite dimensional space are equivalent, this lemma follows immediately.
\end{proof}
\begin{thm}\label{B1}
    For any simplex T, it holds
    \begin{equation}
        \norm{\bmnabla^b_w\mu}_T \sim h_T^{-1}\norm{\mu}_{h, \partial T},~\forall \mu \in M(\partial T).
    \end{equation}
\end{thm}
\begin{proof}	
    In view of $T = A\widehat T+b$,  we have
    \begin{displaymath}
        \begin{split}
            \norm{\bmnabla^b_w\mu}_T 
            &\sim h_T^{\frac{d}{2}} \norm{\widehat{\bmnabla^b_w\mu}}_{\widehat T}\\
            &\sim h_T^{\frac{d}{2}}\norm{A^{-T}\widehat\bmnabla^b_w\widehat\mu}_{\widehat T}~~~~~~~~~~~~~~~~~~~~~~~
            \text{(by Lemma \ref{transformation})}\\
            &\sim h_T^{\frac{d}{2}-1}\norm{\widehat\bmnabla^b_w\widehat\mu}_{\widehat T}~~~~~~~~~~~~~~~~~~~~~~~~~~
            \text{(by \eqref{invATx})}\\
            &\sim h_T^{\frac{d}{2}-1}\norm{\widehat\mu}_{\partial\widehat T}~~~~~~~~~~~~~~~~~~~~~~~~~~~~~
            \text{(by Lemma \ref{append_lbj})}\\
            &\sim h_T^{-\frac{1}{2}}\norm{\mu}_{\partial T}\\
            &\sim h_T^{-1}\norm{\mu}_{h,\partial T}.
        \end{split}
    \end{displaymath}
\end{proof}
Similarly, we can easily prove the following theorem.
\begin{thm}
    For any simplex T, it holds
    \begin{equation}
        \norm{\bmnabla^i_w v}_T \sim h_T^{-1}\norm{v}_T,~\forall v \in V(T).
    \end{equation}
\end{thm}
\begin{lem}\label{appen_lxy}
    For any simplex $T$,      there exist two positive constants $c_T$ and $C_T$ that only depend on T and k, such that 
    \begin{equation}
        c_T(\norm{v}_T + \norm{\mu}_{\partial T}) 
        \leqslant \norm{\bmnabla_w(v,\mu)}_T \leqslant 
        C_T(\norm{v} + \norm{\mu}_{\partial T}),~\forall (v,\mu)\in\Sigma(T),
    \end{equation}
    where 
     $   \Sigma(T) := \{(v,\mu)\in V(T)\times M(\partial T): m_T(\mu) = 0\}.$
\end{lem}
\begin{proof}
    It's easy to know 
    \begin{displaymath}
        (v, \mu) \mapsto \norm{v}_T + \norm{\mu}_{\partial T},~\forall (v, \mu) \in \Sigma(T)
    \end{displaymath}
    defines a norm on $\Sigma(T)$.

    Next we  show  \begin{displaymath}
        (v, \mu) \to \norm{\bmnabla_w(v, \mu)}_T,~\forall (v, \mu) \in \Sigma(T)
    \end{displaymath}
    also defines a norm on $\Sigma(T)$.
    In fact, if $\norm{\bmnabla_w(v, \mu)}_T = 0$, then by the definition of $\bmnabla_w$, i.e. \eqref{grad-v-mu} 
    we have
    \begin{equation}
        (\bmnabla v, \bm q)_T + \langle \mu - v, \bm q \cdot \bm n\rangle_{\partial T} = 0,
        ~\forall \bm q \in\bm W(T).
    \end{equation}
   This relation, together with the properties of the BDM elements (\cite{BrezziDouglasMarini1985}) and the RT elements \cite{RT}, shows
$        v = \mu = \text{constant}.$ Thus the relation 
      $m_T(\mu) = 0$   leads to $(v, \mu) = 0$. 
   
   Finally, the desired conclusion follows from the equivalence of the above two norms.
\end{proof}

\begin{lem}\label{append_lbj_lxy}
    For any simplex T, it holds
    \begin{equation}
        \norm{\bmnabla_w(v,\mu)}_T \sim h_T^{-1}\norm{v}_T + h_T^{-\frac{1}{2}}\norm{\mu}_{\partial T},
        ~\forall (v,\mu)\in\Sigma(T).
    \end{equation}
\end{lem}
\begin{proof}
    In light  of  $T = A\widehat T+b$ and  $m_T(\mu) = m_{\widehat T}(\widehat\mu)$ for all $\mu\in L^2(\partial T)$, we obtain
    \begin{displaymath}
        \begin{split}
            \norm{\bmnabla_w(v, \mu)}_T 
            &\sim h_T^{\frac{d}{2}}\norm{\widehat{\bmnabla_w(v,\mu)}}_{\widehat T}\\
            &\sim h_T^{\frac{d}{2}}\norm{A^{-T}\widehat\bmnabla_w(\hat v,\hat\mu)}_{\widehat T}
            ~~~~~~~~~~~~~~~~~~~~\text{(by Lemma \ref{transformation})}\\
            &\sim h_T^{\frac{d}{2}-1}\norm{\widehat\bmnabla_w(\hat v,\hat\mu)}_{\widehat T}
            ~~~~~~~~~~~~~~~~~~~~~~~\text{(by \eqref{invATx})}\\
            &\sim h_T^{\frac{d}{2}-1}(\norm{\widehat v}_{\widehat T} + \norm{\widehat \mu}_{\partial\widehat T})
            ~~~~~~~~~~~~~~~~~~~~~\text{(by Lemma \ref{appen_lxy})}\\
            &\sim h_T^{-1}\norm{v}_T + h_T^{-\frac{1}{2}}\norm{\mu}_{\partial T}.
        \end{split} 
    \end{displaymath}
  \end{proof}
\begin{thm}\label{B3}
    For any simplex T, it holds
    \begin{equation}
        \norm{\bmnabla_w(v, \mu)}_T \sim h_T^{-1}\norm{v-m_T(\mu)}_T + |{\mu}|_{h, \partial T},
        ~\forall (v,\mu)\in V(T)\times M(\partial T).
    \end{equation}
\end{thm}
\begin{proof}
    By \eqref{grad-v-mu} we have 
    \begin{displaymath}
        \begin{split}
            (\bmnabla_w(v, \mu),\bm q)_T 
            &= -(v,div~\bm q)_T+\langle\mu,\bm q\cdot \bm n\rangle_{\partial T}\\
            &= -(v - m_T(\mu),div~\bm q)_T + \langle\mu-m_T(\mu),\bm q\cdot\bm n\rangle_{\partial T}\\
            &=(\bmnabla_w(v-m_T(\mu), \mu - m_T(\mu)),\bm q)_T,~\forall \bm q \in\bm W(T), 
        \end{split}
    \end{displaymath}
    which implies
    \begin{displaymath}
        \bmnabla_w(v, \mu) = \bmnabla_w(v - m_T(\mu), \mu - m_T(\mu)).
    \end{displaymath}
    Thus it follows
    \begin{displaymath}
        \begin{split}
            \norm{\bmnabla_w(v,\mu)}_T 
            &= \norm{\bmnabla_w(v - m_T(\mu), \mu - m_T(\mu))}_T\\
            &\sim h_T^{-1}\norm{v - m_T(\mu)}_T + h_T^{-\frac{1}{2}}\norm{\mu - m_T(\mu)}_{\partial T}
            ~~~~\text{(by Lemma \ref{append_lbj_lxy})}\\
            &\sim h_T^{-1}\norm{v - m_T(\mu)}_T + |{\mu}|_{h, \partial T}.
        \end{split}
    \end{displaymath}
  \end{proof}
    A combination of Theorems \ref{B1}-\ref{B3} proves Lemma \ref{lem_basic_inequalities}.


\end{document}